\documentclass[leqno 12pts]{amsart}
\usepackage{amsmath, amsthm, amscd, amsfonts, amssymb, graphicx, color}
\usepackage[bookmarksnumbered, colorlinks, plainpages]{hyperref}
\usepackage{tikz}

\newtheorem{theorem}{Theorem}[section]
\newtheorem{lemma}[theorem]{Lemma}
\newtheorem{proposition}[theorem]{Proposition}

\newtheorem{corollary}[theorem]{Corollary}
\newtheorem{definition}[theorem]{Definition}

\usepackage{comment}
\excludecomment{mysection}

\begin{document}
	
\title[An antidote to Kadison's anti-commutant theorem]{Orthogonality: An antidote to Kadison's anti-lattice theorem}
\author{Anil Kumar Karn}
	
\address{School of Mathematical Sciences, National Institute of Science Education and Research, HBNI, Bhubaneswar, P.O. - Jatni, District - Khurda, Odisha - 752050, India.}

\email{\textcolor[rgb]{0.00,0.00,0.84}{anilkarn@niser.ac.in}}
	
	
	
	
\begin{abstract}
	In this paper, we propose  non-commutative analogues of infimum and supremum with the help of algebraic orthogonality.
\end{abstract}

\maketitle 

\section{Introduction} 
Order structure is an essential component of a C$^*$-algebra theory. Using the order structure as well as order theoretic techniques, Gelfand and Naimark proved in \cite{GN} that a norm closed, self-adjoint sub algebra of $B(H)$ can be characterized as a C$^*$-algebra where $H$ is a complex Hilbert space. For the commutative case, they proved that a unital commutative C$^*$-algebra can be identified, up to isometrically *-isomorphism, with $C(X, \mathbb{C})$ for a suitable compact, Hausdorff space $X$. 

In 1941, Kakutani characterized a unital $AM$-space, up to isometric lattice isomorphism, as $C(X, \mathbb{R})$ for a suitable compact, Hausdorff space $X$ \cite{SK}. Thus we note that the self-adjoint part of a (unital) commutative C$^*$-algebra is a vector lattice.

In 1951, Kadison proved that if $H$ is a complex Hilbert space and if $S$ and $T$ are bounded self-adjoint operators on $H$, then $\inf \lbrace S, T \rbrace$ exists in $B(H)_{sa}$ if and only if $S$ and $T$ are comparable \cite{Kad 51}. This is known as Kadison's anti-lattice theorem. In particular, a vector lattice structure can not be expected in a general C$^*$-algebra. This observation contrasts with Kakutani Theorem. 

In this short note, we make an attempt to establish that the `anti-lattice' situation is not completely `anti' lattice. In fact, we show that orthogonality among positive elements can bring a structure in a general C$^*$-algebra which has a close correlation with the vector lattice structure. This paper may be seen as a prequel of \cite{K16, K18}.

\section{A substitute for infimum and supremum} 

 Algebraically orthogonal pairs of positive elements play an important role in the theory of C$^*$-algebras. For example, it follows from the functional calculus that every self-adjoint element $a \in A_{sa}$ has a unique decomposition: $a = a^+ - a^-$ in $A^+$, with $a^+ a^- = 0$. By the functional calculus again, we also get $\vert a \vert = a^+ + a^-$. For $a, b \in A^+$, we say that $a$ is \emph{algebraically orthogonal} to $b$ if $a b = 0$.
 
Recall that the algebraic orthogonality can be defined to a pair of general elements of a C$^*$-algebra. Here, we revisit this notion in the light of its order theoretic characterization. Let us begin with 
\begin{lemma}\label{1}
	Let $A$ be a C$^*$-algebra and let $a b = 0$ for some $a, b \in A^+$. Then $c d = 0$ whenever $0 \le c \le a$ and $0 \le d \le b$.
\end{lemma}
\begin{proof}
	Let $0 \le c \le a$. Then $0 \le b c b \le b a b = 0$ so that $b c b = 0$. It follows that $\Vert c^{\frac{1}{2}} b \Vert^2 = \Vert b c b \Vert = 0$. Thus $c^{\frac{1}{2}} b = 0$ so that $c b = 0$. Now, by the same arguments, we may further conclude that $c d = 0$ whenever $0 \le c \le a$ and $0 \le d \le b$.
\end{proof}
\begin{proposition}\label{2}
	Let $A$ be a C$^*$-algebra. For $a, b \in A_{sa}$, the following statements are equivalent:
	\begin{enumerate}
		\item $\vert a \vert \vert b \vert = 0$;
		\item $a^+, a^-, b^+, b^-$ are mutually algebraically orthogonal;
		\item $\vert a \pm b \vert = \vert a \vert + \vert b \vert$.
	\end{enumerate}
\end{proposition} 
This result is proved in a more general set up using order theoretic techniques. Here we provide a C$^*$-algebraic proof. 
\begin{proof}
	(1) implies (2): Let $\vert a \vert \vert b \vert = 0$. As $0 \le a^+, a^- \le  \vert a \vert$ and $0 \le b^+, b^- \le \vert b \vert$, a repeated use of Lemma \ref{1} yields that $a^+, a^-, b^+, b^-$ are mutually algebraically orthogonal.
	
	(2) implies (1): Let $a^+, a^-, b^+, b^-$ be mutually algebraically orthogonal. Then $(a^+ + a^-) (b^+ + b^-) = 0$. That is, $\vert a \vert \vert b \vert = 0$.
	
	(2) implies (3): Again, let $a^+, a^-, b^+, b^-$ be mutually algebraically orthogonal. Then 
	$$(a^+ + b^+) (a^- + b^-) = 0$$
	and 
	$$(a^+ + b^-) (a^- + b^+) = 0.$$
	
	Thus 
	\begin{eqnarray*}
		\vert a + b \vert &=& \vert a^+ - a^- + b^+ - b^- \vert \\ 
		&=& \vert (a^+ + b^+) - (a^- + b^-) \vert \\
		&=& a^+ + b^+ + a^- + b^- = \vert a \vert + \vert b \vert
	\end{eqnarray*}
	
	and 
	\begin{eqnarray*}
		\vert a - b \vert &=& \vert a^+ - a^- - b^+ + b^- \vert \\ 
		&=& \vert (a^+ + b^-) - (a^- + b^+) \vert \\
		&=& a^+ + b^- + a^- + b^+ = \vert a \vert + \vert b \vert.
	\end{eqnarray*}
	
	(3) implies (2): Finally, assume that $\vert a \pm b \vert = \vert a \vert + \vert b \vert$. Then as before, we may get that $(a^+ + b^+) \perp^a (a^- + b^-)$ and $(a^+ + b^-) \perp^a (a^- + b^+) $. Thus $a^+, a^-, b^+, b^-$ are mutually algebraically orthogonal.
\end{proof}
We `define' that $a, b \in A_{sa}$ are \emph{algebraically orthogonal}, if $\vert a \vert \vert b \vert = 0$. More generally, algebraic orthogonality of a general pair of elements in $A$ can be `described' in terms of an algebraically orthogonal pair of self-adjoint elements in the C$^*$-algebra $M_2(A)$. Let $a, b \in A$. Then $a$ is algebraically orthogonal to $b$, if $\begin{bmatrix} 0 & a \\ a^* & 0 \end{bmatrix}$ is algebraically orthogonal to $\begin{bmatrix} 0 & b \\ b^* & 0 \end{bmatrix}$ in $M_2(A)_{sa}$. Note that $\left\vert \begin{bmatrix} 0 & x \\ x^* & 0 \end{bmatrix} \right \vert = \begin{bmatrix} \vert x^* \vert & 0 \\ 0 & \vert x \vert \end{bmatrix}$ for any $x \in A$. Thus $a$ is algebraic orthogonal to $b$ if and only if $\vert a \vert \vert b \vert = 0$ and $\vert a^* \vert \vert b^* \vert = 0$. Now the following result relates this `definition' with the standard definition of algebraic orthogonality. 
\begin{proposition}\label{3}
	Let $A$ be a C$^*$-algebra and let $a, b \in A$. Then $a^{\ast} b = 0$ if and only if $\vert a^{\ast} \vert \vert b^{\ast} \vert = 0$. In other words, $a$ is algebraically orthogonal to $b$ if and only if $ a  b^* = 0 = a^* b$. In particular, for $a, b \in A_{sa}$, we have $a$ is algebraically orthogonal to $b$ if and only if $a b = 0$.
\end{proposition} 
The following proof was suggested to the author by Antonio M. Peralta.
\begin{proof}
	Without any loss of generality, we may assume that $\Vert a \Vert = 1 = \Vert b \Vert$. First, we assume that $a^{\ast} b = 0$. Then $a a^{\ast} b b^{\ast} = 0$ so that $\vert a^{\ast} \vert \vert b^{\ast} \vert = 0$.
	
	Conversely, assume that $\vert a^{\ast} \vert \vert b^{\ast} \vert = 0$. Then for any $m, n \in \mathbb{N}$ we have $\vert a^{\ast} \vert^{\frac{1}{m}} \vert b^{\ast} \vert^{\frac{1}{n}} = 0$. Since $x^{\frac{1}{m}} \to r(x)$ in the SOT in $A^{\ast\ast}$ for any $x \in A^{\ast\ast +}$, we may conclude that $r(\vert a^{\ast}\vert) r(\vert b^{\ast}\vert) = 0$. Thus 
	$$a^{\ast} b = a^{\ast} r(\vert a^{\ast}\vert) r(\vert b^{\ast}\vert) b = 0.$$
	Now, the other statements follow easily from here.
\end{proof}
This confirms with the traditional definition: Let $a, b \in A$. We say that $a$ is \emph{algebraically orthogonal} to $b$, if $a b^* = 0 = a^* b$. In this case, we write $a \perp^a b$.

Now, we present a result which generalizes the notions of infimum and supremum.
\begin{theorem}\label{4}
	Let $A$ be a C$^*$-algebra and let $a, b \in A_{sa}$.
	\begin{enumerate}
		\item There exists a unique $c \in A_{sa}$ such that 
		\begin{enumerate}
			\item $c \le a, c \le b$; and 
			\item $(a - c) \perp^a (b - c)$.
		\end{enumerate} 
		\item There exists a unique $d \in A_{sa}$ such that 
		\begin{enumerate}
			\item $a \le d, b \le d$; and 
			\item $(d - a) \perp^a (d - b)$.
		\end{enumerate} 
	\end{enumerate}
\end{theorem}
\begin{proof}
	(1). Put $a - b = x$. Then $x \in A_{sa}$. By the functional calculus, there exist unique $x^+, x^- \in A^+$ with $x^+ x^- = 0$ such that $x = x^+ - x^-$ and $\vert x \vert := (x^2)^{\frac{1}{2}} = x^+ + x^-$. Set $c = \frac{1}{2} (a + b - \vert a - b \vert)$. Then $c \in A_{sa}$, 
	$$a - c = \frac{1}{2} (a - b + \vert a - b \vert) = x^+ \in A^+,$$
	and 
	$$b - c = \frac{1}{2} (b - a + \vert a - b \vert) = x^- \in A^+$$ 
	so that $(a - c)(b - c) = x^+ x^- = 0$. Thus $(a - c) \perp^a (b - c)$. 
	
	Next, let $c_1 \in A_{sa}$ such that $c_1 \le a, c_1 \le b$; and $(a - c_1)(b - c_1) = 0$. Put $a - c_1 = a_1$ and $b - c_1 = b_1$. Then $a_1, b_1 \in A^+$ with $a_1 b_1 = 0$. Also $a_1 - b_1 = a - b = x$. Thus by the functional calculus, we get $a_1 = x^+$ and $b_1 = x^-$. Now, it follows that  
	$$c = a - x^+ = a - a_1 = c_1.$$
	Similarly, (2) can be proved by dual arguments.
\end{proof} 
\begin{definition}
	Let $A$ be a C$^*$-algebra and let $a, b \in A_{sa}$. We define $a \dot\wedge b := \frac{1}{2} (a + b - \vert a - b \vert)$ as the \emph{ortho-infimum} of $a$ and $b$. Similarly, we define $a \dot\vee b := \frac{1}{2} (a + b + \vert a - b \vert)$ as the \emph{ortho-supremum} of $a$ and $b$.
\end{definition} 
We note that these notions coincide with the usual notions of infimum and supremum respectively, in the case of a vector lattice. Let $(L, L^+, \wedge, \vee)$ be a vector lattice. Recall that a pair of elements $x, y \in L$ is said to be orthogonal, (we write, $x \perp^{\ell} y$), if $\vert x \vert \wedge \vert y \vert = 0$. Here $\vert u \vert := u \vee (-u)$ for all $u \in L$. 
\begin{corollary}\label{5}
	Let $V$ be a vector lattice and let $x, y \in V$.
	\begin{enumerate}
		\item There exists a unique element $u = x \wedge y \in V$ such that 
		\begin{enumerate}
			\item $u \le x$, $u \le y$; and 
			\item $(x - u) \perp^{\ell} (y - u)$.
		\end{enumerate} 
		\item There exists a unique element $v = x \vee y \in A_{sa}$ such that 
		\begin{enumerate}
			\item $x \le v$, $y \le v$; and 
			\item $(v - x) \perp^{\ell} (v - y)$.
		\end{enumerate} 
	\end{enumerate}
\end{corollary} 
\begin{proof}
	Note that
	$$\frac{1}{2} (x + y - \vert x - y \vert) = x \wedge y$$ 
	and 
	$$\frac{1}{2} (x + y + \vert x - y \vert) = x \vee y.$$ 
	Thus proof of Theorem \ref{4} can be replicated.
\end{proof} 
Thus, in a vector lattice, the notions ortho-infimum and ortho-supremum coincide with the notions of infimum and supremum, respectively. Narrating differently, the notions of ortho-infimum and ortho-supremum extend the notions of infimum and supremum respectively (confirmed in vector lattices) to the self-adjoint part of a general C$^*$-algebra.

\section{orthogonality and norm} 
Recall that a vector lattice $V$ with a norm $\Vert\cdot\Vert$ is called an $AM$-space, if 
\begin{enumerate}
	\item $\Vert  u \Vert \le \Vert v \Vert$ whenever $u, v \in V$ with $\vert u \vert \le \vert v\vert$; and 
	\item  $\Vert u \vee v \Vert = \max \lbrace \vert u \Vert, \Vert v \Vert \rbrace$ for all $u, v \in V^+$.
\end{enumerate} 
 A positive element $e \in V^+$ is called an \emph{order unit} for $V$, if for each $v \in V$, there exists $k > 0$ such that $k e \pm v \in V^+$. If, in addition, 
$$\Vert v \Vert = \inf \lbrace k > 0: k e \pm v \in V^+ \rbrace$$ 
for all $v \in V$, then $V$ is called a unital $AM$-space. 
\begin{definition}
	Let $V$ be a real normed linear space. For $u, v \in V$ we say that $u$ is \emph{$\infty$-orthogonal} to $v$, (we write, $u \perp _{\infty} v$), if $ \left\Vert u + kv \right\Vert = \max \{ \Vert u \Vert, \Vert kv \Vert \},$ for all $k \in \mathbb{R}$. 
\end{definition}
\begin{definition}
	Let $A$ be a C$^*$-algebra. For $a, b \in A^+$, we say that $a$ is \emph{absolutely $\infty$-orthogonal} to $b$, (we write $a \perp_{\infty}^a b$), if $c \perp_{\infty} d$ whenever $0 \le c \le a$ and $0 \le d \le b$. 
\end{definition}
If $V$ is an $AM$-space, then $V$ is isometrically order isomorphic to $A_{sa}$ for some commutative C$^*$-algebra $A$. Thus the notion of absolute $\infty$-orthogonality makes sense in $V^+$ as well. In fact, absolute $\infty$-orthogonality can be defined in a more general set up \cite[Definition 4.5]{K16}. It was proved in \cite[Theorem 2.1]{K18}, (see also \cite[Theorem 4.3 and Conjecture 4.4]{K16}), that in a C$^*$-algebra $A$, we have $\perp^a = \perp_{\infty}^a $ on $A^+$. For an $AM$-space, can have a stronger result.

\begin{proposition}\label{6} 
	In an $AM$-space $V$, we have $\perp^{\ell} = \perp_{\infty}^a$ on $V^+$. 
\end{proposition} 
Though this result can be deduced from \cite[Theorem 4.2]{K16} using Kakutani's theorem for $AM$-spaces, we give a direct order theoretic proof.
\begin{proof}
	Let $u, v \in V^+$. First, we assume that $u \perp{\ell} v$, that is, $u \wedge v = 0$. Then for $0 \le u_1 \le u$, $0 \le v_1 \le v$ and any $k \in \mathbb{R}$, we have 
	$$\vert u_1 \vert \wedge \vert k v_1 \vert = u_1 \wedge (\vert k \vert v_1) = 0$$ 
	so that 
	$$\vert u_1 + k v_1 \vert = u_1 + \vert k \vert v_1 = u_1 \vee (\vert k \vert v_1).$$
	Since $V$ is an $AM$-space, we get that 
	$$\Vert u_1 + k v_1 \Vert = \Vert \vert u_1 + k v_1 \vert \Vert = \Vert u_1 \vee (\vert k \vert v_1) \Vert = \max \lbrace \Vert u_1 \Vert, \Vert k v_1 \Vert \rbrace.$$
	Thus $u \perp_{\infty}^a v$. 
	
	Conversely, assume that $u \perp_{\infty}^a v$. Put $w = u \wedge v$. Then $0 \le w \le u$ and $w \le v$. Thus, by assumption, $w \perp_{\infty} w$. But then $w = 0$. Hence $u \perp^{\ell} v$.
\end{proof}
If we write $\perp$ as a generalization of $\perp^a$ in a C$^*$-algebra and $\perp^{\ell}$ in a vector lattice, we may deduce the following definition. (See \cite{K18} with special attention to Remark 3.3.) 
\begin{definition}
	Let $(V, V^+)$ be a real ordered vector space. Assume that $\perp$ is a binary relation in $V$ such that for $u, v, w \in V$, we have
	\begin{enumerate}
		\item $u \perp 0$;
		\item $u \perp v$ implies $v \perp u$; 
		\item $u \perp v$ and $u \perp w$ imply $u \perp (k v + w)$ for all $k \in \mathbb{R}$; 
		\item For each $u \in L$, there exist unique $u^+, u^- \in L^+$ with $u^+ \perp u^-$ such that $u = u^+ - u^-$. \\
		Let us put $u^+ + u^- := \vert u \vert$.
		\item If $u \perp v$ and if $\vert w \vert \le \vert v \vert$, then $u \perp w$. 		
	\end{enumerate} 
	Then $V$ is called an \emph{absolutely ordered vector space}.
\end{definition}
A normed version of this notion is included in the following result.
\begin{theorem}\label{7}
	Let $(V, e)$ be an order unit space. Then the following two sets of conditions are equivalent:
	\begin{enumerate}
		\item 
		\begin{enumerate}
			\item For each $u \in V$, the exists a unique pair $u^+, u^- \in V^+$ with $u^+ \perp_{\infty}^a u^-$ such that $u = u^+ - u^-$; \\			
			Set $|u| := u^+ + u^-$. 
			\item If $u, v, w \in V^+$ with $u \perp_{\infty}^a v$ and $u \perp_{\infty}^a w$, then we have $u \perp_{\infty}^a |v \pm w|$.
		\end{enumerate}
		\item $V$ is an absolutely ordered vector space in which $\perp = \perp_{\infty}^a$ on $V^+$. 
	\end{enumerate}
\end{theorem} 
In this case, $(V, |\cdot|, e)$ is called an \emph{absolute order unit space}. again, we may recall that a more general normed version of absolutely ordered vector spaces was introduced in \cite[Definition 3.8]{K18}.
\begin{proof}
	Clearly, (2) implies set (1). Let us now assume that (1)(a) and (1)(b) hold. For $u, v \in V$, we define $u \perp v$, if $\vert u \vert \perp_{\infty}^a \vert v \vert$. Then $\perp = \perp_{\infty}^a$ on $V^+$. Also then the following fact follow immediately from the definition of $\perp_{\infty}^a$:
	\begin{enumerate}
		\item[(i)] For each $u \in V$, there exist unique $u^+, u^- \in V^+$ with $u^+ \perp u^-$ such that $u = u^+ - u^-$;
		\item[(ii)] $u \perp 0$ for all $u \in V$;
		\item[(iii)] $u \perp v$ implies $v \perp u$; 
		\item[(iv)] $u \perp v$ implies $u \perp k v$ for any $k \in \mathbb{R}$;  and 
		\item[(v)] if $u \perp v$ and $\vert w \vert \le \vert v \vert$, then $u \perp w$.
	\end{enumerate} 
	Thus it only remains to prove that $u \perp (v + w)$ whenever $u \perp v$ and $u \perp w$. Assume that $u \perp v$ and $u \perp w$. Then $\vert u \vert \perp_{\infty}^a \vert v \vert$ and $\vert u \vert \perp_{\infty}^a \vert w \vert$. Then by the definition of $\perp_{\infty}^a$, we may deduce that $\vert u \vert \perp_{\infty}^a \lbrace v^+, v^-, w^+, w^- \rbrace$. Now, by a repeated use of (1)(b), we obtain that $\vert u \vert \perp_{\infty}^a \vert v + w \vert$ for 
	$$v + w = v^+ - v^- + w^+ - w^- = (v^+ + w^+) - (v^- + w^-).$$
	Thus $u \perp (v + w)$. Hence $V$ is an absolutely ordered vector space. 
\end{proof}
In some recent works, we have been able to extend some of the properties of operator algebras to absolute (matrix) order unit spaces \cite{K18, KK19, KK19a}. We hope to present absolute (matrix) order unit spaces as a non-commutative analogue of unital $AM$-spaces. 

\thanks{{\bf Acknowledgements}: The author is thankful to Antonio M. Peralta for his inputs on orthogonality in C$^*$-algebras.}

\end{document}